\documentclass[12pt,a4paper,reqno]{amsart}

\usepackage{amssymb}
\usepackage{enumerate}
\usepackage[all]{xy}
\newtheorem{theorem}{Theorem}
\theoremstyle{definition}
\newtheorem{proposition}{Proposition}[section]
\newtheorem{lemma}[proposition]{Lemma}

\theoremstyle{definition}

\theoremstyle{remark}
\newtheorem{remark}[proposition]{Remark}

\newcommand{\PP}{\mathbb{P}}

\newcommand{\pr}{\mathrm{pr}}
\newcommand{\Z}{\mathbb{Z}}

\newcommand{\C}{\mathbb{C}}
\newcommand{\N}{\mathbb{N}}
\newcommand{\GG}{\mathbb{G}}

\newcommand{\cO}{\mathcal O}
\newcommand{\Sp}{\mathrm{Spec}}
\newcommand{\Sing}{\mathrm{Sing}}
\newcommand{\Gr}{\mathrm{Gr}}
\newcommand{\gr}{\mathrm{gr}}
\newcommand{\sth}{\,\,|\,\,}

\newcommand{\longto}{\longrightarrow}
\renewcommand{\epsilon}{\varepsilon}

\renewcommand{\div}{\mathrm{div}}

\title[Russell's hypersurface from a geometric point of view]{Russell's hypersurface from a geometric point of view}

\author{Isac Hed\'en}
  
\address{Isac Hed\'en\\
 Mathematisches Institut\\
Universit\"at Basel\\
 Rheinsprung 21\\
  4051 Basel, Switzerland}
\email{Isac.Heden@unibas.ch}

\thanks{This work was done as part of PhD studies at the Department of Mathematics, Uppsala University; the support from the Swedish graduate school in Mathematics and Computing (FMB) is gratefully acknowledged. Many thanks also go to Karl-Heinz Fieseler for his supervision. Finally, I am thankful to the reviewer for many valuable comments and references, especially concerning the historical background of the problem.}
\date{\today}

\begin{document}

\begin{abstract}The famous Russell hypersurface is a smooth complex affine threefold which is diffeomorphic to a euclidean space but not algebraically isomorphic to the three dimensional affine space. This fact was first established by Makar-Limanov, using algebraic minded techniques. In this article, we give an elementary argument which adds a greater insight to the geometry behind the original proof and which also may be applicable in other situations.\end{abstract}


\maketitle
\section{Introduction}
Russell's hypersurface
$$
X:=\{(x,y,z,t) \in \C^4\sth x+x^2y+z^3+t^2=0\} \hookrightarrow \C^4,
$$ is one of the most prominent examples of an exotic variety, i.e. a variety which is diffeomorphic to an affine space (\cite{ChDi94}, \cite[Lemma 5.1]{Ka05}), but not isomorphic to it. The latter is an immediate consequence of Theorem~\ref{mlmthm} below, which states that there are not sufficiently many actions of the additive group $\GG_a$ on $X$, and the aim of this paper is to give an elementary argument for this theorem. It includes some important elements of the original proof, but gives a greater geometrical insight to the situation.\par

The study of exotic varieties goes back to a paper of Ramanujam \cite{Ra71}, where a nontrivial example of a topologically contractible smooth affine algebraic surface $S$ over $\C$ is constructed. Ramanujam observed that $S\times\C$ is diffeomorphic to $\C^3$, and asked whether this product is also isomorphic to $\C^3$. This was later proven not to be the case, and thus the algebraic structure on $\C^3$ coming from $S\times\C^2$ is exotic \cite{Zai93}. Later on, many other exotic structures on $\C^3$ have been constructed, see e.g. the introduction of \cite{Za98} for a list. Note also that there are no exotic structures on affine space in dimension $\leq 2$ \cite{Ra71}.\par

The motivation for studying Russell's hypersurface originally came from the linearization conjecture for $\C^3$, which claims that each $\GG_m$-action on $\C^3$ is linearizable. In the proof of this result by Koras and Russell, they described a list of smooth affine threefolds diffeomorphic to $\C^3$ which contains all the potential counterexamples to the conjecture, and thus it was reduced to determining whether all of these so called Koras-Russell threefolds are exotic \cite{KoRu97}. Kaliman and Makar-Limanov established exoticity for some of them \cite{KML97}, and Russell's hypersurface is the ''most simple'' among the remaining ones. The difficulty with Russell's hypersurface was that all the usual algebraic and geometric invariants failed to distinguish it from $\C^3$. Makar-Limanov finally established exoticity of Russell's hypersurface (Theorem~\ref{mlmthm}), and later on Kaliman and Makar-Limanov were able to prove exoticity of the remaining Koras-Russell threefolds \cite{KaML97,KaML07} as well, elaborating on Makar-Limanov's methods. This confirmed the linearization conjecture \cite{KKML97}.\par
From now on, we will focus on Makar-Limanov's result, stated in the following theorem.
\begin{theorem}[Makar-Limanov, \cite{LML96}]\label{mlmthm}
The projection $\pr_1:X \longrightarrow \C,\, (x,y,z,t) \mapsto x$ is invariant with respect to any $\GG_a$-action on $X$.
\end{theorem}
Some years after Makar-Limanov proved Theorem~\ref{mlmthm}, Kaliman proved, using non-elementary birational geometry, that morphisms ${\C^3 \longrightarrow \C}$ with generic fiber $\C^2$ cannot have any other fibers \cite{Ka02}. Since all the fibers of $\pr_1:X\longrightarrow\C$ are $\C^2$ except the zero fiber $\pr_1^{-1}(0)$, it follows also from Kaliman's result that $X \not\cong \C^3$. In 2005, Makar-Limanov gave another proof of the exoticity of Russell's hypersurface \cite{MLM05}; yet another proof was given by Derksen \cite{Der97}, and Crachiola also proved the exoticity in the positive characteristic case \cite{Cra06}. The original proof of Theorem~\ref{mlmthm} used algebraic techniques, while we rather focus on a geometric approach using fibrations and quotient maps.

\par

\medskip

\noindent
 {\bf An outline of our proof}.
In order to prove Theorem~\ref{mlmthm}, we make use of an isomorphism $X \cong U \subset M$ with an open subset $U$ of a blowup 
$\pi : M \longrightarrow \C^3$, such that $D:=M \setminus U$
is the strict transform of $\{0\} \times \C^2 \hookrightarrow \C^3$ and
$$
X \cong U \subset M \longrightarrow \C^3
$$
is the map $(x,y,z,t) \mapsto (x,z,t)$. That is, $X$ is isomorphic to an affine modification $U$ of $\C^3$. The key-result is then that $\cO (M) \subset \cO(X)$ is invariant for any $\GG_a$-action on $X$. Since $\cO(M) \cong \cO (\C^3)$, this allows us to conclude that for any given $\GG_a$-action on $X$, there is an induced $\GG_a$-action on $\C^3$ which makes $\pi|_X:X \longrightarrow \C^3$ equivariant. Then $\pi(U)$ is obviously invariant, and it follows that its interior $\C^*\times\C^2$ is invariant as well. Theorem~\ref{mlmthm} is obtained from this by observing that any $\GG_a$-action on $\C^* \times \C^2$ leaves the first coordinate invariant: a nontrivial $\GG_a$-orbit is isomorphic to $\C$, but there are no non-constant morphisms from $\C$ to $\C^*$.

\section{Russell's hypersurface in a blowup of $\C^3$}
We recall the realization of Russell's hypersurface as an affine modification of $\C^3$, see also \cite[Example 1.5]{KaZa99}. Let $N \hookrightarrow \C^2 = \Sp (\C[z,t])$ denote the affine cuspidal cubic curve given by
$$
N:= \{ (z,t) \in \C^2\sth z^3+t^2=0 \},
$$ and
let $I:=(g,h)\subset\C[x,z,t]$ denote the ideal which is generated by the two relatively prime polynomials 
$g(x,z,t)=x^2$ and $h(x,z,t)=x+z^3+t^2$. The zero set of $I$ is $\{0\}\times N$, and the blowup 
$$
M:=Bl_I({\C^3}) \cong \{ ((x,z,t), [u:v]) \in \C^3 \times \PP^1\sth h(x,z,t)u+g(x,z,t)v=0 \}
$$
of $\C^3$ along $I$
is a hypersurface in $\C^3 \times \PP^1$ with singular locus of codimension two:
$
\Sing(M)=\{0\} \times N \times \{[0:1]\}.$ In particular, $M$ is a normal variety. 

\begin{remark}
Russell's hypersurface is isomorphic to the open subset $U$ of $M$ given by $u\neq 0$, via the embedding $X\hookrightarrow M,\,\,(x,y,z,t)\mapsto ((x,z,t),[1:y])$. 
\end{remark}
We denote the complement of $U$ in $M$ by $D$. Note that $D$ is then given by $u=0$, and that the image of $U$ under the blowup morphism is $\pi (U) = \C^* \times \C^2 \cup ( \{0\} \times N).$
\medskip

\section{Additive group actions on Russell's hypersurface}
In order to see that $\cO(M) \subset A:=\cO(X)$ is invariant for every $\GG_a$-action on $X$, we show the equivalent fact that $\cO(M) \subset A$ is stable under every locally nilpotent derivation $\partial: A \longrightarrow A$. This obviously holds for the trivial $\GG_a$-action on $X$, so we may assume that $\partial \not=0$. The first step is to characterize $\cO(M)$ in terms of a filtration on $A$. 

\begin{remark}
With the filtration
$$A_{\le n}:=  \cO_{nD}(M)= \{ f \in \C(M)^*\sth \div(f) \ge - nD \}\cup\{0\}$$ we have $A_{\le 0}=\cO (M)=\pi^*(\cO(\C^3))$, so $\pi^*(\cO(\C^3))$ is stable with respect to a locally nilpotent derivation $\partial:A\longrightarrow A$ if and only if $\partial(A_{\leq 0})\subseteq A_{\leq 0}$. 
\end{remark}

In order to understand the above filtration, we treat $A$ as a subset of $\C(x,z,t)$ and note that multiplicities along $D$ are simply multiplicities along $\{0\} \times \C^2$; so $x,y,z$ and $t$ have multiplicities $1,-2,0$ and $0$, respectively.
Now
every element $f\in A \setminus \{0\}$ can be written in the form 
$$
f=\sum_{i=0}^k y^ip_i(x,z,t),
$$
 where each $p_i$ is at most linear in $x$ for $i\geq 1$ (and $p_k\neq 0$).
Thus
$$
A=\bigoplus_{k=-\infty}^\infty A_k
$$ 
is a direct sum of free $\C[z,t]$-modules $A_k$ of rank 1 defined as
\begin{equation}\label{aenn}
 A_k:=\left\{
\begin{array}{lcl}
 \C[z,t]x^{|k|}&,&\textit{ if } k\leq 0\\
 \C[z,t]y^\ell&,&\textit{ if } k=2\ell>0\\
 \C[z,t]xy^\ell&,&\textit{ if } k=2\ell-1>0
\end{array}\right. ,
\end{equation}
and one can check that 
$$
A_{\leq n}=\bigoplus_{k\leq n} A_k.
$$ 
We thus obtain an explicit description of the associated graded algebra 
$$
B:=\Gr(A) = \bigoplus_{n\in\Z}B_n\ \ \text{with}\ B_n:=A_{\leq n}\slash A_{\leq n-1}.
$$ 
It is generated by the elements $\gr (x) \in B_{-1}, \gr (y) \in B_2,\gr (z),\gr(t) \in B_0$ and
$$
W:=\Sp(B)\cong \{(x,y,z,t)\in\C^4\sth x^2y+z^3+t^2=0\}.
$$
In particular $B_0=\C[\gr(z),\gr(t)]\simeq\C[z,t]$.
\begin{remark}
This grading was also used by M. Zaidenberg, see \cite[Lemma 7.4]{Za98}
\end{remark}

Let $\ell=\ell(\partial)\in\Z$ be minimal with the property that $\partial(A_{\le n}) \subset A_{\le n+\ell}$ for all $n\in\Z$; the existance of such an $\ell$ follows from the fact that both $A_{\le 0}$ and $B$ are finitely generated graded algebras, and $\ell\neq -\infty$ since we consider a nontrivial $\GG_a$-action. It follows that $\partial:A\longrightarrow A$ induces a nontrivial homogeneous locally nilpotent derivation on $B$ of degree $\ell$; we will denote it by $\delta$. With this notation it is enough to show that $\ell\leq 0$ in order to obtain $\partial(A_{\geq 0})\subseteq A_{\geq 0}$. In fact, more is true:

\begin{proposition}\label{keyprop}
\label{homograd} With $B$ as above, any nontrivial locally nilpotent homogeneous derivation 
$\delta : B \longrightarrow B$ has degree $\ell < 0$.
\end{proposition}

Before going into the proof, let us start with a discussion of the geometry of $W \hookrightarrow \C^4$ and prove Lemma~\ref{1or2} below.
As a hypersurface in $\C^4$, $W$ is a normal variety since its singular set $\Sing(W)= \{0\} \times \C \times \{0\} \times \{0\}$ has codimension two. It admits two different group actions: 
the $\GG_a$-action $(\tau, w) \mapsto \tau .w$ corresponding to the locally nilpotent derivation $\delta:B\longto B$; and the $\GG_m$-action corresponding to the grading of $B$. The latter is given by
$$\GG_m\times W\longto W,\quad(\lambda,(x,y,z,t)) \mapsto (\lambda^{-1} x,\lambda^2 y,z,t),$$ and since $B_0 = \C[z,t]$, the $\GG_m$-quotient morphism is given by $$p:W \rightarrow \C^2\cong\Sp(\C[z,t]),\quad (x,y,z,t)\mapsto (z,t).$$
It is trivial above $\C^2\setminus N$: the map
$$
(\C^2\setminus N)\times\GG_m \stackrel\sim\longto p^{-1}(\C^2\setminus N),\quad
((z,t),\lambda) \mapsto \left(\lambda^{-1}, -(z^3+t^2)\lambda^2, z, t\right),
$$
is a $\GG_m$-equivariant isomorphism with inverse
$$
p^{-1}(\C^2\setminus N) \stackrel\sim\longto (\C^2\setminus N)\times\GG_m,\quad
(x,y,z,t) \mapsto ((z,t),x^{-1}).
$$
 As for $N$, we have 
 $p^{-1}(N)=F_-\cup F_+,$ where $F_-$ and $F_+$ are the subsets of $p^{-1}(N)$ given by $y=0$ and $x=0$ respectively.
\begin{remark}
The set $F_-$ consists exactly of the points $w\in W$ for which $\lim_{\lambda\to\infty}\lambda w$ exists, and $F_+$ consists exactly of the points $w\in W$ for which $\lim_{\lambda\to 0}\lambda w$ exists.
\end{remark}
\begin{remark}\label{trivoutsideN}
The above trivialization extends to a trivialization $\C^2\times\GG_m\stackrel\sim\longrightarrow W \setminus F_+,$ but for $W \setminus F_-$ there is no such trivialization since the $\GG_m$-isotropy group of a point in $F_+ \setminus F_-$ has order 2.
\end{remark}

Now let us turn to the $\GG_a$-action $\GG_a \times W \longrightarrow W, (\tau, w) \mapsto \tau .w$, corresponding to $\delta: B \longrightarrow B$. Since $\delta$ is homogeneous of degree $\ell$, it is normalized by the $\GG_m$-action, i.e. for $w\in W$, $\tau \in \GG_a$ and $\lambda \in \GG_m$, we have 
$$(\lambda^{-\ell}\tau).(\lambda w)=\lambda(\tau.w).$$

In particular this implies that $\lambda O$ is a $\GG_a$-orbit for any $\GG_a$-orbit $O$.

\begin{lemma}\label{1or2}
Let $\delta: B \longrightarrow B$ be a nontrivial locally nilpotent derivation, homogeneous of degree $\ell$. Then either
\begin{enumerate}
\item $\ell < 0$ and $F_+$ is invariant, or
\item $\ell > 0$ and $F_-$ is invariant.
\end{enumerate}
\end{lemma}
\begin{proof}

 Since the locally nilpotent derivation $\delta:B\rightarrow B$ is homogeneous, its kernel 
$$
B^\delta:=\{f\in B\sth \delta(f)=0\}=\{f\in B\sth f(\tau.w)=f(w)\,\,\forall \tau\in \GG_a\}
$$
 is a graded subalgebra, i.e.:
$$
B^\delta=\bigoplus_{n\in\Z}B_n^\delta.
$$ 
Given $f \in B_k \setminus \{0\}$, we have $\delta^\nu f \in B_{k+\nu \ell}^\delta \setminus \{0\}$ for a 
suitable $\nu \in \N$. It follows that
\begin{enumerate}
\item
if $\ell=0$, we have $B_n^\delta \not=\{0\}$ for all $n\not=0$, 
\item
if $\ell > 0$ we have $B_n^\delta \not=\{0\}$ for some $n >0$,
\item
if $\ell <0$ we have $B_n^\delta \not=\{0\}$ for some  $n<0$.
\end{enumerate}

First assume that $\ell>0$, so that $B_n^\delta\neq 0$ for some $n$, and let $f\in B_n^\delta \setminus \{0\}$. Then $f$ vanishes on $F_-$ since $f(\lambda x)= \lambda^n f(x)$ and since $\lim_{\lambda \to \infty}\lambda x$ exists in $W$ for $x \in F_-$. It follows that $F_-$ is invariant since it is an irreducible component of the invariant set $V(f)\subset W$ of dimension two. If $\ell<0,$ it follows analogously that $F_+\subset W$ is invariant. It remains to show that $\ell$ cannot be zero.\par

If $\ell=0$, both $F_-$ and $F_+$ are invariant. So $p^{-1}(N)$ is invariant and $W\setminus p^{-1}(N)$ as well. Then for any nontrivial $\GG_a$-orbit $O\subset W\setminus p^{-1}(N)$  the map $(z^3+t^2)\circ p|_O$ has no zeros, and thus must be constant, say with value $a\in\C^*$, since $O \cong \C$. However, any morphism $p|_O:O\rightarrow V(\C^2;z^3+t^2-a)$ from the complex line to the smooth affine elliptic curve $V(\C^2;z^3+t^2-a)$ is constant, so $O$ is contained in a $p$-fiber. Since $p(O)\in\C^2\setminus N$, this $p$-fiber is isomorphic to $\GG_m$, as $p$ is a $\GG_m$-principal bundle over $\C^2\setminus N$. This gives a contradiction since $\C$ cannot be embedded into $\GG_m$.\end{proof}

\begin{proof}[Proof of Proposition~$\ref{keyprop}$]
By Lemma~\ref{1or2} it is enough to show show that $F_-,$ given by $y=0$, is not invariant. Suppose to the contrary that $F_-$ is invariant; then its complement in $W$, given by $y\neq 0$, is invariant as well. Since there is no non-constant invertible function on $\GG_a$-orbits, all $\GG_a$-orbits in $W\setminus F_-$ are contained in level hypersurfaces of $y$. In particular the hypersurface $V\subset W$ which is given by $y=1$ is invariant and we have $V\simeq\{(x,y,z)\in\C^3\sth x^2+z^3+t^2=0\}.$
The restriction of the $\GG_m$-quotient projection  
$$
\psi:=p|_V:V\rightarrow \C^2,\quad(x,z,t)\mapsto (z,t)
$$ 
is a two sheeted branched covering of $\C^2$ with
branch locus $\psi^{-1}(N)\cong N$ and deck transformation 
$$
\sigma:V\rightarrow V,\quad (x,z,t)\mapsto(-x,z,t),
$$ which is simply the action of $-1\in\GG_m$. In particular $\sigma (O)$ 
is a $\GG_a$-orbit for any $\GG_a$-orbit $O$. Assume for the moment that every nontrivial $\GG_a$-orbit intersects $\psi^{-1}(N)$ exactly once. Since the hypersurface $V$ is a normal surface, there is a quotient map 
$$
\chi:V\longto V//\GG_a:=\Sp (\cO (V)^{\GG_a}),
$$ 
the generic fiber of which is a $\GG_a$-orbit \cite[Lemma 1.1]{Fie94}. Thus the restriction 
$$
\chi|_{\psi^{-1}(N)}:\psi^{-1}(N)\longto V//\GG_a
$$ 
is injective on a nonempty open subset of $\psi^{-1}(N)$. Hence, $V//\GG_a$ being a smooth curve, it follows from Zariski's main theorem that this restriction is an open embedding. However, this is a contradiction since the affine cuspidal cubic curve $\psi^{-1}(N)$ has a singular point.

Finally, any nontrivial $\GG_a$-orbit $O$ intersects $\psi^{-1}(N)$: otherwise $\psi|_O$ would be a non-constant morphism to an affine elliptic curve $z^3+t^2=a$ for some $a\in \C^*$, which is impossible. Note that a point in $O\cap\psi^{-1}(N)$ is a common point of the two $\GG_a$-orbits $\sigma(O)$ and $O$, so $\sigma(O)=O$. Choose an equivariant isomorphism $\C\cong O$ such that $0\in\C$ corresponds to a point in $\psi^{-1}(N)$. Then the involution $\sigma:O\longto O$ corresponds to $\C\longto\C,\zeta\mapsto -\zeta$, and as a consequence every nontrivial $\GG_a$-orbit $O\hookrightarrow V$ intersects the branch locus $\psi^{-1}(N)=W^\sigma$ (the fixed point set of $\sigma$) in exactly one point.
\end{proof}

\end{document}